\documentclass[12pt]{amsart}
\usepackage[utf8]{inputenc}
\usepackage{amsmath}
\usepackage{amsfonts}
\usepackage{amssymb}
\usepackage{amsthm}
\usepackage{bbm}
\usepackage{hyperref}
\renewcommand{\sqrt}[1]{\left( #1 \right)^\frac12}

\newcommand{\e}[0]{\mathrm{e}}
\newcommand{\E}[0]{\mathbb{E}}
\newcommand{\T}[0]{\mathbb{T}}
\newcommand{\N}[0]{\mathbb{N}}

\newcommand{\R}[0]{\mathbb{R}}
\newcommand{\Z}[0]{\mathbb{Z}}

\renewcommand{\d}[0]{\mathrm{d}}
\newcommand{\sgn}[0]{\mathrm{sgn}}

\renewcommand{\coprod}{\bigsqcup}

\newcommand{\p}[1]{\left(#1\right)}

\newcommand{\id}{\mathrm{id}}
\newcommand{\honeall}{H^1_{\mathrm{all}}}
\newcommand{\honelast}{H^1_{\mathrm{last}}}
\newcommand{\geqlast}{\geq_{\mathrm{last}}}

\newcommand{\beq}{\begin{equation}}
\newcommand{\eeq}{\end{equation}}
\newcommand{\supp}[0]{\mathrm{supp}\,}
\newcommand{\clspan}[0]{\overline{\mathrm{span}}}

\author{Maciej Rzeszut}
\email{maciej.rzeszut@gmail.com}

\title{Disjointification inequalities for Hoeffding subspaces of $H^1$ on an infinite polydisc}
\newtheorem{thm}{Theorem}
\newtheorem{lem}[thm]{Lemma}

\newtheorem{cor}[thm]{Corollary}
\newtheorem{prop}[thm]{Proposition}

\numberwithin{thm}{section}
\numberwithin{equation}{section}

\allowdisplaybreaks
\begin{document}
\begin{abstract}
We prove that the $L^1$ norm on the linear span of functions on $\T^\N$ dependent on $m$ variables and analytic and mean zero in each of them can be expressed as an interpolation sum of $H^1\p{\T^S,\ell^1\p{\N^S,H^2\p{\T^{[1,m]\setminus S},\ell^2\p{\N^{[1,m]\setminus S}}}}}$ norms over $S\subseteq [1,m]$ and derive some interpolation consequences. 
\end{abstract}
\maketitle
\section{Introduction}
We will use $\T:=\R/\Z$ as a model for the circle group. For $I\subset \N$ (finite or infinite), the following subspaces of $L^1\p{\T^I}$ will be of interest.
\begin{align} U^1_m\p{\T^I}&=\clspan\left\{\e^{2\pi i \langle n,t\rangle}:\left|\supp n\right|=m\right\}\\
H^1_{\mathrm{all}}\p{\T^I}&= \clspan\left\{\e^{2\pi i\langle n,t\rangle}:n_j\geq 0\text{ for all }j\in I\right\}\\
H^1_{\mathrm{last}}\p{\T^I}&= \clspan\left\{\e^{2\pi i\langle n,t\rangle}:n_j\geq 0\text{ for }j=\max\supp n\right\}\end{align}
In \cite{diss} we proved a decomposition theorem for functions in $U^1_m\p{\T^\N}$. Here we are going to prove that the summands of the decomposition can be chosen to be analytic if the function belongs to $U^1_m\p{\T^\N}\cap \honeall\p{\T^\N}$. \par
Let us understand the case $m=1$ first. There are two fundamental facts that give the desired result when put together. 

\begin{thm}[Johnson, Schechtman \cite{JSch}] 
Let $f_i\in L^1\p{\T}$ be of mean zero. Then there exist $g_i\in L^1\p{\T}$, $h_i\in L^2\p{\T}$ of mean zero such that $f_i=g_i+h_i$ and
\beq \label{eq:intrjsch} \int_{\T^\N}\d x\left|\sum_i f_i\p{x_i}\right| \simeq \int_{\T^\N}\d x\sqrt{\sum_i \left|f_i\p{x_i}\right|^2}\gtrsim \sum_i \int_\T \d \xi \left|g\p{\xi}\right| +\sqrt{\sum_i \int_\T \d \xi \left|h_i\p{\xi}\right|^2}.\eeq
\end{thm}

The fact that the rightmost expression depends only on $H^1$ norm of $g_i$ and $H^2$ norm of $h_i$ allows us to plug the decomposition into the scalar valued particular case of the following.

\begin{thm}[Pisier \cite{pisierinterp}, Bourgain \cite{bourginterp}]
The couple 
\beq \p{H^1\p{\T,\ell^2},H^2\p{\T,\ell^2}}\eeq
is $K$-closed in
\beq \p{L^1\p{\T,\ell^2},L^2\p{\T,\ell^2}}.\eeq
\end{thm}

Therefore, for any $f(x)=\sum_i f_i\p{x_i}$ with $f_i\in H^1\p{\T}$ of mean zero, there are $g_i, h_i\in H^1\p{\T}$ of mean zero such that $f_i=g_i+h_i$ and
\beq \left\|f\right\|_{U^1_1\p{\T^\N}}\gtrsim \sum_i \left\|g_i\right\|_{H^1\p{\T}} + \sqrt{\sum_i \left\|h_i\right\|_{H^2\p{\T}}^2}.\eeq
In the multivariate case, we do not hope for a similar expression in terms of norms of analytic functions of $m$ variables. Instead, we will find a decomposition of $U^1_m$ norm into an interpolation sum of norms that involve only integration over at most $m$ variables at a time. Our result is obtained by means of decoupling methods and iteration of univariate version and the bulk of the article is devoted to developing necessary machinery to make the induction step possible. \par
The main tool is an extension of a particular case of a theorem of Kislyakov and Xu, concerning interpolation of lattice-valued Hardy spaces. 

In Section 2, we will point out the techincal differences necessary to transfer the original proof to our case of $H^1_{\mathrm{last}}$. We also note two following known propositions. The first one is a straightforward vector-valued analogue of the known square function theorem for $L^1$ Hardy martingales and the second is meant to alleviate measurability problems.
\begin{prop}\label{diagproj}
Let $\E_{\{n\}}$ denote the orthogonal projection from $L^2\p{\T^\N}$ onto functions which depend only on the $n$-th coordinate and $\E_{\{n\}}\otimes \id$ the natural extension to $L^2\p{\T^\N,\mathcal{H}}$. Then the map
\beq \p{f_n}_{n\in \N}\mapsto \p{\p{\E_{\{n\}}\otimes \id}f_n}_{n\in \N}\eeq
is bounded on $\honelast\p{\T^\N,\ell^2\p{\N,\mathcal{H}}}$.
\end{prop}

\begin{prop}\label{trivsumL1}If $S$ is a measure space and $\p{X_i}$ is a family of Banach spaces, then 
\beq L^1\p{S, \sum_i X_i} \subseteq \sum_i L^1\p{S,X_i}\eeq
(the other inclusion is trivial).
\end{prop}
In Section 3, we prove the main result. In Section 4, we show how to use the main result to derive some interpolation properties of $U^1_m\cap\honeall$.

\section{Theorem of Kislyakov and Xu for $\honeall\p{\T^\N}$}
The Pontryagin dual to the countable product $\T^\N$ is the infinite direct sum $\Z^{\oplus\N}$, consisting of sequences of integers, only finitely many of which are nonzero. We will equip this group with the total order $\geqlast$, where the positive cone consists of sequences $n$ such that $n_{j_0}>0$, where $j_0$ is the biggest $j$ such that $n_j\neq 0$. The sign of $n\in \Z^{\oplus\N}$ will be taken with respect to this order. \par
The classical Hilbert transform $\mathfrak{H}$ is defined on $L^2\p{\T}$ by the formula 
\beq \widehat{\mathfrak{H} f}(n)=-i\sgn n \widehat{f}(n)\eeq
for $n\in \Z$. The same forumula defines $\mathfrak{H}$ on $L^2\p{\T^\N}$. These definitions are consistent in the sense that $\mathfrak{H}$ on $\T^\N$ restricted to functions dependent only on the first variable coincides with $\mathfrak{H}$ on $\T$. It is known that for any $1<p<\infty$, $\mathfrak{H}\otimes \id_X$ is bounded on $L^p\p{\T,X}$ if $X$ is a UMD space. We will make use of an extension of this fact to $\T^\N$. 
\begin{lem}If $X$ is a $UMD$ space and $1<p<\infty$, then $\mathfrak{H}\otimes \id_X$ is bounded on $L^p\p{\T^\N,X}$.\end{lem}
\begin{proof}
Let us fix a polynomial $\varphi\in L^p\p{\T^\N,X}$ and let $A=\supp \widehat{\varphi}$. While the finite set $A$ is fixed, by the transference principle we can choose a sequence $\alpha_j$ increasing fast enough so that the operator
\beq L^p_A\p{\T^n,X}\ni f=\sum_{n\in A} \widehat{f}(n)\e^{2\pi i \p{\sum_j n_j t_j}}\mapsto M_\alpha f=\sum_{n\in A} \widehat{f}(n) \e^{2 \pi i \tau\sum_j n_j \alpha_j}\in L^p\p{\T,X}\eeq
is bounded and maps positive characters on $\T^\N$ (with respect to the order we chose on $\Z^{\oplus\infty}$) to positive characters on $\T$. Since both $f$ and $\p{\mathfrak{H}\otimes \id_X}f$ have Fourier transform supported on $A$, we have
\begin{eqnarray} \left\|\p{\mathfrak{H}\otimes \id_X} f\right\|_{L^p\p{\T^\N,X}}
&\simeq& \left\|M\p{\mathfrak{H}\otimes \id_X}f\right\|_{L^p\p{\T,X}}  \\
&=& \left\|\p{\mathfrak{H}\otimes \id_X}\p{Mf}\right\|_{L^p\p{\T,X}}\\
&\lesssim_{p,X}& \left\|Mf\right\|_{L^p\p{\T,X}}\\
&\simeq&  \left\|f\right\|_{L^p\p{\T^\N,X}}\end{eqnarray}\end{proof}
The original proof of K-X theorem \cite{KX} patched by this lemma is enough to obtain its $\T^\N$ version.
\begin{thm}[Kislyakov, Xu+$\varepsilon$]
Let $E_0$, $E_1$ be Banach sequence lattices. Assume that there exists $\alpha$ such that convexifications $E_0^{(\alpha)}$, $E_1^{(\alpha)}$ are UMD. Then for any $1\leq p_0,p_1<\infty$, the couple 
\beq \p{H^{p_0}_{\mathrm{last}}\p{\T^\N, E_0}, H^{p_1}_{\mathrm{last}}\p{\T^\N, E_1}}\eeq
is $K$-closed in 
\beq \p{L^{p_0}\p{\T^\N, E_0}, L^{p_1}\p{\T^\N, E_1}}.\eeq
\end{thm}

Let $A,B$ be finite sets. If $f_i\in \honeall\p{\T^{A}\times \T^B,\mathcal{H}}$ for $i\in \N^A\times \N^B$, we define a norm on the sequence $\p{f_i}$ by
\beq \left\|\p{f_i}_{i\in \N^A\times \N^B}\right\|_{E_{A,B}\p{\mathcal{H}}}= \sum_{i_A\in \N^A} \int_{\T^A} \d \xi_A \sqrt{ \sum_{i_{B}\in \N^{B}}\int_{\T^{B}}\d\xi_{B}\left\|f_{i_A,i_{B}}\p{\xi_A,\xi_{B}}\right\|_\mathcal{H}^2}.\eeq
In other words,
\beq E_{A,B}\p{\mathcal{H}}= \ell^1\p{\N^A,\honeall\p{\T^A,\ell^2\p{\N^B,H^2_{\mathrm{all}}\p{\T^B,\mathcal{H}}}}}.\eeq

By copying the proof of Xu's result form \cite{xul1h1} with additional regard to the fact that the unconditional basis of choice in $H^1\p{\T}$ can also be chosen to be orthonormal \cite{wojtaszczyk}, we get the following.

\begin{lem}[Xu+$\varepsilon$]Let $\p{\chi_k}_{k\in\Z}$ be an unconditional orthonormal basis of $H^1\p{\T}$ and denote by $\chi_{k_1,\ldots,k_m}= \otimes_{j=1}^m \chi_{k_j}$ the elements of the corresponding unconditional orthonormal basis of $\honeall\p{\T^m}$. Then, for $A\dot\cup B=[1,m]$, the map
\beq \p{f_i}_{i\in \N^m}\mapsto \p{ \left\langle f_i , \chi_k \otimes e_s \right\rangle }_{i\in \N^m, k\in \Z^m,s\in S},\eeq
where $\p{e_s}_{s\in S}$ is an orthonormal basis of $\mathcal{H}$, is an isomorphism between $E_{A,B}\p{\mathcal{H}}$ and a Banach lattice, the 2-convexification of which is UMD. 
\end{lem}
As a straighforward consequence, we get 

\begin{cor}\label{kcleab} The family 
\beq \p{\honelast \p{\T^\N, E_{A,B}\p{\mathcal{H}}}}_{A\dot\cup B=[1,m]}\eeq
is $K$-closed in 
\beq \p{L^1\p{\T^\N, E_{A,B}\p{\mathcal{H}}}}_{A\dot\cup B=[1,m]}.\eeq
\end{cor}

\begin{cor}\label{kclh1last}The couple 
\beq \p{\honelast\p{\T^\N,\ell^1\p{H^1\p{\T,\ell^2}}}, \honelast\p{\T^\N,\ell^2\p{H^2\p{\T,\ell^2}}}}\eeq
is $K$-closed in 
\beq \p{L^1\p{\T^\N,\ell^1\p{H^1\p{\T,\ell^2}}}, L^1\p{\T^\N,\ell^2\p{H^2\p{\T,\ell^2}}}}.\eeq
\end{cor}

\begin{cor}\label{onemorecor}
The couple 
\beq \p{ \honeall\p{\T^m,H^1\p{\T,\ell^2}}, \honeall\p{\T^m,H^2\p{\T,\ell^2}} }\eeq
is $K$-closed in 
\beq \p{ L^1\p{\T^m,H^1\p{\T,\ell^2}}, L^1\p{\T^m,H^2\p{\T,\ell^2}} }\eeq
\end{cor}

\section{Decomposition theorems}
Let $f\in U^1_m\p{\T^\N}$ have a representation
\beq f(x)=\sum_{i_1<\ldots<i_m}f_{i_1,\ldots,i_m}\p{x_{i_1},\ldots,x_{i_m}}.\eeq
By Bourgain's square function theorem \cite{bourgwalsh} for Hoeffding subspaces of $L^1$ and Zinn's decoupling ineqaulity \cite{zinn},
\begin{eqnarray} \|f\|_{L^1}&\simeq& \int_{\T^\N}\sqrt{\sum_{i_1<\ldots<i_m}\left|f_{i_1,\ldots,i_m}\p{x_{i_1},\ldots,x_{i_m}}\right|^2}\d x\\
\label{eq:aftersqdecoup}&\simeq& \int_{\p{\T^\N}^m}\sqrt{\sum_{i_1<\ldots<i_m}\left|f_{i_1,\ldots,i_m}\p{x^{(1)}_{i_1},\ldots,x^{(m)}_{i_m}}\right|^2}\d x^{(1)}\ldots \d x^{(m)}
\end{eqnarray}
and therefore we will work with expressions of the form \eqref{eq:aftersqdecoup}.

\begin{thm}Let $f_{i,j}\in\honeall\p{\T^2}$. Then
\begin{align}\int_{\T^\N}\int_{\T^\N}\sqrt{\sum_{i,j}\left|f_{i,j}\p{x_i,y_j}\right|^2}\d x\d y& \\
\simeq\inf_{\substack{f_{i,j}=a_{i,j}+b_{i,j}+c_{i,j}+d_{i,j}\\ a_{i,j},b_{i,j},c_{i,j},d_{i,j}\in\honeall\p{\T^2}}} &  \sum_{i,j}\int_\T\int_\T\left|a_{i,j}\p{\xi,\upsilon}\right|\d\xi\d\upsilon\\
+& \sqrt{\sum_{i,j}\int_\T\int_\T\left|b_{i,j}\p{\xi,\upsilon}\right|^2\d\xi\d\upsilon}\\
+&\sum_i\int_\T\sqrt{\sum_j\int_\T \left|c_{i,j}\p{\xi,\upsilon}\right|^2\d \upsilon}\d\xi\\
+&\sum_j\int_\T\sqrt{\sum_i\int_\T \left|d_{i,j}\p{\xi,\upsilon}\right|^2\d \xi}\d\upsilon.\end{align}
\end{thm}
\begin{proof}
The ``$\leq$'' inequality is true with constant 1 and proved as in \cite{diss}. Define a sequence of functions $F_j:\T^\N\times \T\to\ell^2$ by 
\beq F_j\p{x,\upsilon}=\p{f_{i,j}\p{x_i,\upsilon}}_{i\in\N}.\eeq
Suppose that 
\beq 1= \int_{\T^\N}\int_{\T^\N}\sqrt{\sum_{i,j}\left|f_{i,j}\p{x_i,y_j}\right|^2}\d x\d y= \int_{\T^\N}\int_{\T^\N}\sqrt{\sum_{j}\left\|F_{j}\p{x,y_j}\right\|_{\ell^2}^2}\d x\d y.\eeq
For fixed $x$, we apply an $\ell^2$-valued Johnson-Schechtman inequality to functions $F_j\p{x,\cdot}$ and obtain a decomposition
\beq F_j=G_j+H_j,\quad G_j\p{x,\upsilon}= \p{g_{i,j}\p{x,\upsilon}}_{i\in \N},\quad H_j\p{x,\upsilon}= \p{h_{i,j}\p{x,\upsilon}}_{i\in \N}\eeq
(note that we lose analyticity and the information that $i$-th coordinates depend only on $x_i$) such that
\beq \label{eq:wecould}\int_{\T^\N}\sqrt{\sum_{j}\left\|F_{j}\p{x,y_j}\right\|_{\ell^2}^2}\d y\gtrsim \sum_j \int_\T \left\|G_j\p{x,\upsilon}\right\|_{\ell^2}\d\upsilon + \sqrt{\sum_j \int_\T \left\|H_j\p{x,\upsilon}\right\|_{\ell^2}^2\d\upsilon }\eeq
for all $x$. Since $F_j$'s were analytic in the second variable, by Pisier-Bourgain theorem, we can improve the decomposition to additionally satisfy 
\beq G_j\p{x,\cdot}=\p{g_{i,j}\p{x,\cdot}}_{i\in \N}\in H^1\p{\T,\ell^2},\quad H_j\p{x,\cdot}=\p{h_{i,j}\p{x,\cdot}}_{i\in\N}\in H^2\p{\T,\ell^2}.\eeq
Integrating \eqref{eq:wecould} with respect to $x$, we get
\beq \label{eq:havehad} 1\gtrsim \int_{\T^\N}\d x\sum_j\int_\T \d\upsilon\sqrt{\sum_i\left|g_{i,j}\p{x,\upsilon}\right|^2}+  \int_{\T^\N}\d x\sqrt{\sum_j\int_\T \d\upsilon\sum_i\left|h_{i,j}\p{x,\upsilon}\right|^2},\eeq
where 
\beq \p{g_{i,j}}_{i,j\in\N}\in L^1\p{\T^\N,\ell^1\p{H^1\p{\T,\ell^2}}}, \quad \p{h_{i,j}}_{i,j\in\N}\in L^1\p{\T^\N,\ell^2\p{H^2\p{\T,\ell^2}}}\eeq
and
\beq \label{eq:itall} g_{i,j}\p{x,\upsilon}+h_{i,j}\p{x,\upsilon}=f_{i,j}\p{x_i,\upsilon}.\eeq
Here, the inner $\ell^2$ is indexed by $i$ and the outer $\ell^1,\ell^2$ are indexed by $j$. A careful reader might notice that we are cheating a bit: there is no guarantee that $g_{i,j},h_{i,j}$ are measurable on the product space $\T^\N\times \T$, as they were constructed for each $x\in\T^\N$ separately. This can be alleviated as follows. Let $\mathcal{D}_N$ be the product of the $N$-th dyadic sigma algebra on first $N$ coordinates of $\T^\N$ and the trivial one on the rest. The inequality 
\begin{align} &\int_{\T^\N}\int_{\T^\N}\sqrt{\sum_{i,j}\left|f_{i,j}\p{x,y_j}\right|^2}\d x\d y\gtrsim \\
&\inf_{g_{i,j}\p{x,\upsilon}+h_{i,j}\p{x,\upsilon}= f_{i,j}\p{x_i,\upsilon}} \left\|\p{g_{i,j}}_{i,j}\right\|_{L^1\p{\T^\N,\ell^1\p{H^1\p{\T,\ell^2}}}}+ \left\|\p{h_{i,j}}_{i,j}\right\|_{L^1\p{\T^\N,\ell^2\p{H^2\p{\T,\ell^2}}}}\end{align}
is, by the above reasoning, true for $f_{i,j}$'s on $\T^\N\times \T$ that are $\mathcal{D}_N$-measurable in the first variable and analytic in the second, for some $N$. Such functions are dense in all three norms appearing in the inequality, so it holds true whenever the left hand side is finite and $f_{i,j}$ are in $H^1\p{\T}$ with respect to the second variable.\par
We are in position to use Corollary \ref{kclh1last}, which, by the fact that $\T^\N\ni x\mapsto f_{i,j}\p{x_i,\upsilon}$ is in $\honelast\p{\T^\N}$ for each $i,j,\upsilon$, allows us to enforce
\beq \p{g_{i,j}}_{i,j\in\N}\in \honelast\p{\T^\N,\ell^1\p{H^1\p{\T,\ell^2}}}, \quad \p{h_{i,j}}_{i,j\in\N}\in \honelast\p{\T^\N,\ell^2\p{H^2\p{\T,\ell^2}}}\eeq
while maintaining \eqref{eq:havehad} and \eqref{eq:itall}. 
Now, since $f_{i,j}\p{x_i,\upsilon}$ depends on $x$ only through $x_i$, \eqref{eq:itall} gives
\beq \p{\E_{\{i\}}\otimes \id} g_{i,j}+ \p{\E_{\{i\}}\otimes \id} h_{i,j} =f_{i,j}.\eeq
Thus, the functions $\tilde{g}_{i,j}$, $\tilde{h}_{i,j}$ defined on $\T^2$ by
\beq \p{\E_{\{i\}}\otimes \id} g_{i,j}\p{x,\upsilon}= \tilde{g}_{i,j}\p{x_i,\upsilon}, \quad \p{\E_{\{i\}}\otimes \id} h_{i,j}\p{x,\upsilon}= \tilde{h}_{i,j}\p{x_i,\upsilon}\eeq
satisfy
\beq \tilde{g}_{i,j}+\tilde{h}_{i,j}=f_{i,j}.\eeq
By Lemma \ref{diagproj} applied at each $j,\upsilon$,
\begin{align}
&\int_{\T^\N}\d x\sum_j\int_\T \d\upsilon\sqrt{\sum_i\left| \tilde{g}_{i,j}\p{x_i,\upsilon}\right|^2}\\
=&\int_{\T^\N}\d x\sum_j\int_\T \d\upsilon\sqrt{\sum_i\left| \p{\E_{\{i\}}\otimes \id}g_{i,j}\p{x,\upsilon}\right|^2}\\
\lesssim & \int_{\T^\N}\d x\sum_j\int_\T \d\upsilon\sqrt{\sum_i\left|g_{i,j}\p{x,\upsilon}\right|^2}.
\end{align}
Similarly, by $L^2$-valued Lemma \ref{diagproj},
\begin{align}
&\int_{\T^\N}\d x\sqrt{\sum_j\int_\T \d\upsilon\sum_i\left|\tilde{h}_{i,j}\p{x_i,\upsilon}\right|^2}\\
=&\int_{\T^\N}\d x\sqrt{\sum_j\int_\T \d\upsilon\sum_i\left|\p{\E_{\{i\}}\otimes \id} h_{i,j}\p{x,\upsilon}\right|^2}\\
= & \int_{\T^\N}\d x\sqrt{ \sum_i\left\|\p{\E_{\{i\}}\otimes \id}\coprod_j h_{i,j}\p{x,\cdot}\right\|_{L^2\p{\coprod_j \T}}^2}\\
\lesssim & \int_{\T^\N}\d x\sqrt{ \sum_i\left\|\coprod_j h_{i,j}\p{x,\cdot}\right\|_{L^2\p{\coprod_j \T}}^2}\\
=& \int_{\T^\N}\d x\sqrt{\sum_j\int_\T \d\upsilon\sum_i\left|h_{i,j}\p{x,\upsilon}\right|^2}.
\end{align}
Since $g_{i,j}, h_{i,j}$ were in $\honelast$ in the first variable and respectively $H^1,H^2$ in the second, $\tilde{g}_{i,j},\tilde{h}_{i,j}$ are repsectively in $H^1\p{\T,H^1\p{\T}}$, $H^1\p{\T,H^2\p{\T}}$. Also, by plugging the above inequalities into \eqref{eq:havehad}, they satisfy
\beq \label{eq:youhad} 1\gtrsim \int_{\T^\N}\d x\sum_j\int_\T \d\upsilon\sqrt{\sum_i\left|\tilde{g}_{i,j}\p{x_i,\upsilon}\right|^2}+  \int_{\T^\N}\d x\sqrt{\sum_j\int_\T \d\upsilon\sum_i\left|\tilde{h}_{i,j}\p{x_i,\upsilon}\right|^2}.\eeq
By J-Sch + P/B at each $j,\upsilon$, 
\beq \label{eq:rolling} \int_{\T^\N}\d x\sqrt{\sum_i\left| \tilde{g}_{i,j}\p{x_i,\upsilon}\right|^2}\gtrsim \sum_i \int_\T\d\xi \left|a_{i,j}\p{\xi,\upsilon}\right| + \sqrt{ \sum_i \int_\T\d\xi \left|d_{i,j}\p{\xi,\upsilon}\right|^2 },\eeq
where 
\beq \label{eq:inthe} a_{i,j}+d_{i,j}=\tilde{g}_{i,j}\eeq
and $a_{i,j}, d_{i,j}$ are analytic in the first variable. Integrating \eqref{eq:rolling} with respect to $\upsilon$ and summing over $j$, we get 
\begin{align} &\int_{\T^\N}\d x\sum_j\int_\T\d\upsilon\sqrt{\sum_i\left| \tilde{g}_{i,j}\p{x_i,\upsilon}\right|^2}\gtrsim\\
&\label{eq:deep}\sum_j\int_\T\d\upsilon \sum_i \int_\T\d\xi \left|a_{i,j}\p{\xi,\upsilon}\right| + \sum_j\int_\T\d\upsilon \sqrt{ \sum_i \int_\T\d\xi \left|d_{i,j}\p{\xi,\upsilon}\right|^2 }.\end{align}
Here, again, we prove the above first for $\tilde{g}_{i,j}$ being simple functions in $\upsilon$ and argue by density. For each $j$, we can now invoke the $K$-closedness of
\beq H^1\p{\T, \ell^1\p{H^1\p{\T}}}, H^1\p{\T, \ell^2\p{H^2\p{\T}}}\eeq
in 
\beq L^1\p{\T, \ell^1\p{H^1\p{\T}}}, L^1\p{\T, \ell^2\p{H^2\p{\T}}}\eeq
granted by Corollary \ref{kclh1last} (here, $\xi$ corresponds to the inner $\T$ and $\upsilon$ to the outer) and obtain $a_{i,j},d_{i,j}\in \honeall\p{\T^2}$ satisfying \eqref{eq:inthe} and \eqref{eq:deep}. In order to obtain the other two summands, we apply Hilbert-valued J-Sch+P/B to the sequence of functions $\T\ni\xi\mapsto \p{\tilde{h}_{i,j}\p{\xi,\cdot}}_{j\in \N}\in \ell^2\p{H^2\p{\T}}$ indexed by $i$, which produces a decomposition
\beq c_{i,j}+b_{i,j}=\tilde{h}_{i,j}\eeq
of $\tilde{h}_{i,j}$ into two summands, which retains analyticity in the first variable as well as the space of values, and satisfies
\begin{align} \label{eq:myheart}&\int_{\T^\N}\d x\sqrt{\sum_j\int_\T \d\upsilon\sum_i\left|\tilde{h}_{i,j}\p{x_i,\upsilon}\right|^2}\gtrsim\\
& \sum_i \int_\T\d\xi \sqrt{\sum_j\int_\T\d\upsilon \left|c_{i,j}\p{\xi,\upsilon}\right|^2} + \sqrt{\sum_i \int_\T\d\xi \sum_j\int_\T\d\upsilon \left|b_{i,j}\p{\xi,\upsilon}\right|^2}.
\end{align}
Plugging \eqref{eq:myheart} and \eqref{eq:deep} into \eqref{eq:youhad}, we see that the decomposition $f_{i,j}= a_{i,j}+b_{i,j}+c_{i,j}+d_{i,j}$ satisfies the desired inequality.
\end{proof}

\begin{thm}\label{mainh12msum}Let $\mathcal{H}$ be a separable Hilbert space and $f_{i_1,\ldots,i_m}\in \honeall\p{\T^m,\mathcal{H}}$ for $i=\p{i_1,\ldots,i_m}\in\N^m$. Then
\beq \overbrace{\int_{\T^\N}\cdots\int_{\T^\N}}^{m}\sqrt{\sum_{i}\left\|f_i\p{x^{(1)}_{i_1},\ldots,x^{(m)}_{i_m}}\right\|_\mathcal{H}^2}\d x^{(1)}\ldots \d x^{(m)}\simeq_m \eeq
\beq \inf_{\substack{f_i=\sum_{A\dot\cup B=[1,m]}q^{(A,B)}_i\\ q^{(A,B)}_i\in \honeall\p{\T^m,\mathcal{H}} }}\sum_{A\dot\cup B=[1,m]} \sum_{i_A\in \N^A} \int_{\T^A} \d \xi_A \sqrt{ \sum_{i_{B}\in \N^{B}}\int_{\T^{B}}\d\xi_{B}\left\|q^{(A,B)}_{i_A,i_{B}}\p{\xi_A,\xi_{B}}\right\|_\mathcal{H}^2}.\eeq
\end{thm}
\begin{proof}
The inequality ``$\leq$'' has already been handled in the non-analytic case. For the ``$\gtrsim$'' inequality, we will proceed by induction. The case $m=1$ is just the usual J-Sch+P/B. Let us assume that the theorem is true for some $m$. \par
Suppose $f_{i_0,\ldots,i_m}\in \honeall\p{\T^{m+1},\mathcal{H}}$ for $i_0,\ldots,i_m\in \N$. We define 
\beq F_{i_1,\ldots,i_m}\in \honeall\p{\T^\N\times \T^m,\ell^2\p{\N,\mathcal{H}}}\eeq
by
\beq F_{i_1,\ldots,i_m}\p{x,\xi}=\p{f_{i_0,i_1,\ldots,i_m}\p{x_{i_0},\xi_1,\ldots,\xi_m}}_{i_0\in\N}.\eeq
We have 
\begin{align} &\int_{\p{\T^\N}^{m+1}}\sqrt{\sum_{i_0,\ldots,i_m}\left\|f_{i_0,\ldots,i_m}\p{ x^{(0)}_{i_0},\ldots,x^{(m)}_{i_m}}\right\|_{\mathcal{H}}^2}\d x^{(0)}\ldots\d x^{(m)} \\
=& \int_{\T^\N}\int_{\p{\T^\N}^m}\sqrt{ \sum_{i_1,\ldots,i_m}\left\|F_{i_1,\ldots,i_m}\p{ y,x^{(1)}_{i_1},\ldots,x^{(m)}_{i_m}}\right\|_{\ell^2\p{\N,\mathcal{H}}}^2 }\d x^{(1)}\ldots \d x^{(m)}\d y.\end{align}
For a fixed value of $y$, the inside integration is, by $\ell^2\p{\N,\mathcal{H}}$-valued induction hypothesis ($\N$ corresponds to $i_0$), equivalent to 
\beq \left\| \p{F_i\p{y,\cdot}}_{i\in \N^m} \right\|_{\sum_{A\dot\cup B=[1,m]} E_{A,B}\p{\ell^2\p{\N,\mathcal{H}}}}.\eeq
Each $F_i$ is in $\honeall$ in the first variable (the one in $\T^\N$), so in particular in $\honelast$. Corollary \ref{kcleab} guarantees 
\beq \honelast\p{ \sum_{A\dot\cup B=[1,m]} E_{A,B}\p{\ell^2\p{\N,\mathcal{H}}} } \subseteq \sum_{A\dot\cup B=[1,m]} \honelast\p{E_{A,B}\p{\ell^2\p{\N,\mathcal{H}}}},\eeq
which produces a decomposition
\beq\label{eq:FeqsumG} F_i=\sum_{A\dot\cup B=[1,m]}G^{(A,B)}_i\eeq
of each $F_i$ into summands $G^{(A,B)}_i\in \honelast\p{\T^\N,\honeall\p{\T^m,\ell^2\p{\N,\mathcal{H}}}}$ such that
\begin{align} & \int_{\T^\N}\int_{\p{\T^\N}^m}\sqrt{ \sum_{i_1,\ldots,i_m}\left\|F_{i_1,\ldots,i_m}\p{ y,x^{(1)}_{i_1},\ldots,x^{(m)}_{i_m}}\right\|_{\ell^2\p{\N,\mathcal{H}}}^2 }\d x^{(1)}\ldots \d x^{(m)}\d y\\
=& \left\| \p{F_i}_{i\in \N^m}\right\|_{\honelast\p{ \sum_{A\dot\cup B=[1,m]} E_{A,B}\p{\ell^2\p{\N,\mathcal{H}}} }}\\
\gtrsim & \sum_{A\dot\cup B=[1,m]}\left\|\p{G^{(A,B)}_i}_{i\in \N^m}\right\|_{ \honelast\p{\T^\N,E_{A,B}\p{\ell^2\p{\N,\mathcal{H}}} } }.\end{align}
For $i_0\in \N$, $i_{\geq 1}\in \N^m$, let $g^{(A,B)}_{i_0,i_{\geq 1}}\in \honelast\p{\T^\N,\honeall\p{\T^m,\mathcal{H}}}$ be the $i_0$-th coordinate of $G^{(A,B)}_i$ with respect to the innermost $\N$. By \eqref{eq:FeqsumG},
\beq \sum_{A\dot\cup B=[1,m]}g^{(A,B)}_{i_0,i_{\geq 1}}\p{y,\xi }= f_{i_0,i_{\geq 1}}\p{y_{i_0},\xi}\eeq
for $\p{i_0,i_{\geq 1}}= \p{i_0,\ldots,i_m}\in \N^{m+1}$, $y\in \T^\N$, $\xi\in \T^m$. In particular, the left hand side depends only on the $i_0$-th entry of the coordinate from $\T^\N$. Applying $\E_{\{i_0\}}\otimes \id$ to both sides we get
\beq \label{eq:sumtildeg}\sum_{A\dot\cup B=[1,m]} \tilde{g}^{(A,B)}_{i_0,i_{\geq 1}}\p{\upsilon,\xi }= f_{i_0,i_{\geq 1}}\p{\upsilon,\xi}\eeq
for $\upsilon\in\T$, where $\tilde{g}^{(A,B)}_{i_0,i_{\geq 1}}$ is defined by
\beq \p{\E_{\{i_0\}}\otimes \id}{g}^{(A,B)}_{i_0,i_{\geq 1}}\p{y,\xi }= \tilde{g}^{(A,B)}_{i_0,i_{\geq 1}}\p{y_{i_0},\xi }\eeq
and therefore belongs to $\honeall\p{\T^{m+1},\mathcal{H}}$, because, for fixed $\xi$, $\p{\E_{\{i_0\}}\otimes \id}{g}^{(A,B)}_{i_0,i_{\geq 1}}$ belongs to $\honelast \cap V^1_1\subset \honeall$ as a function of the first variable. Let us denote 
\beq \tilde{\mathcal{H}}_B= \ell^2\p{\N^B,H^2_{\mathrm{all}}\p{\T^B,\mathcal{H}}}.\eeq
For each $i_0,i_A$, we can treat $\p{\tilde{g}_{i_0,i_A,i_B}}_{i_B\in \N^B}=\tilde{g}_{i_0,i_A}$ as an element of $\honeall\p{\T^A,\tilde{\mathcal{H}}_B}$ in a natural way. For fixed $A,B$, by $\tilde{\mathcal{H}}_B$-valued Lemma \ref{diagproj} applied at each $i_A,\xi_A$,
\begin{align} &\left\|\p{G^{(A,B)}_i}_{i\in \N^m}\right\|_{ \honelast\p{\T^\N,E_{A,B}\p{\ell^2\p{\N,\mathcal{H}}} } } \\
= & \int_{\T^\N}\d y \sum_{i_A\in \N^A} \int_{\T^A}\d \xi_A \sqrt{\sum_{i_B\in \N^B}\int_{\T^B}\d \xi_B \sum_{i_0\in \N} \left\|g^{(A,B)}_{i_0,i_A,i_B}\p{y,\xi_A,\xi_B}\right\|_{\mathcal{H}}^2}\\
\gtrsim & \int_{\T^\N}\d y \sum_{i_A\in \N^A} \int_{\T^A}\d \xi_A \sqrt{\sum_{i_B\in \N^B}\int_{\T^B}\d \xi_B \sum_{i_0\in \N} \left\|\p{\E_{\{i_0\}}\otimes\id}g^{(A,B)}_{i_0,i_A,i_B}\p{y,\xi_A,\xi_B}\right\|_{\mathcal{H}}^2} \\
= & \int_{\T^\N}\d y \sum_{i_A\in \N^A} \int_{\T^A}\d \xi_A \sqrt{\sum_{i_B\in \N^B}\int_{\T^B}\d \xi_B \sum_{i_0\in \N} \left\|\tilde{g}^{(A,B)}_{i_0,i_A,i_B}\p{y_{i_0},\xi_A,\xi_B}\right\|_{\mathcal{H}}^2}\\
=&  \sum_{i_A\in \N^A} \int_{\T^A}\d \xi_A \int_{\T^\N}\d y \sqrt{\sum_{i_0\in \N} \left\|\tilde{g}^{(A,B)}_{i_0,i_A}\p{y_{i_0},\xi_A}\right\|_{\tilde{\mathcal{H}}_B}^2}.\end{align}
For fixed $i_A,\xi_A$, the expression beginning with $\int_{\T^\N}$ is the interpolation sum of $\ell^1\p{\N,H^1\p{\T,\tilde{\mathcal{H}}_B}}$ and $\ell^2\p{\N,H^2\p{\T,\tilde{\mathcal{H}}_B}}$ by the usual J-Sch+P/B. Together with Proposition \ref{trivsumL1} this gives a decomposition
\beq \label{eq:gtileqphipsi}\tilde{g}^{(A,B)}_{i_0,i_A}= \varphi^{(A,B)}_{i_0,i_A} + \psi^{(A,B)}_{i_0,i_A}\eeq
such that 
\beq \varphi^{(A,B)}_{i_0,i_A}\in L^1\p{\T^A,H^1\p{\T,\tilde{\mathcal{H}}_B}},\quad \psi^{(A,B)}_{i_0,i_A}\in L^1\p{\T^A,H^2\p{\T,\tilde{\mathcal{H}}_B}}\eeq
and
\begin{align} 
 &\sum_{i_A\in \N^A} \int_{\T^A}\d \xi_A \int_{\T^\N}\d y \sqrt{\sum_{i_0\in \N} \left\|\tilde{g}^{(A,B)}_{i_0,i_A}\p{y_{i_0},\xi_A}\right\|_{\tilde{\mathcal{H}}_B}^2}\\
\gtrsim & \sum_{i_A\in \N^A} \int_{\T^A}\d \xi_A \sum_{i_0\in \N}\int_{\T}\d \upsilon \left\|\varphi^{(A,B)}_{i_0,i_A}\p{\upsilon,\xi_A}\right\|_{\tilde{\mathcal{H}}_B} +\\
 & \sum_{i_A\in \N^A} \int_{\T^A}\d \xi_A \sqrt{\sum_{i_0\in \N}\int_{\T}\d \upsilon \left\|\psi^{(A,B)}_{i_0,i_A}\p{\upsilon,\xi_A}\right\|_{\tilde{\mathcal{H}}_B}^2}.\end{align} 
By Corollary \ref{onemorecor}, $\varphi^{(A,B)}_{i_0,i_A}$ and $\psi^{(A,B)}_{i_0,i_A}$ can be chosen to be in $\honeall$ in the second variable (the one in $\T^A$) as well. Taking 
\beq q^{(A\cup\{0\},B)}_{i_0,i_{\geq 1}}\p{\xi_0,\xi_{\geq 1}}= \p{\varphi^{(A,B)}_{i_0,i_A}\p{\xi_0,\xi_A}}_{i_B}\p{\xi_B},\eeq
\beq q^{(A,B\cup\{0\})}_{i_0,i_{\geq 1}}\p{\xi_0,\xi_{\geq 1}}= \p{\psi^{(A,B)}_{i_0,i_A}\p{\xi_0,\xi_A}}_{i_B}\p{\xi_B}\eeq
gives a decomposition of $f_i$, because 
\beq \sum_{A\dot\cup B=[0,m]}q_i^{(A,B)}=f_i\eeq
by \eqref{eq:sumtildeg} and \eqref{eq:gtileqphipsi}. Moreover, 
\beq \sum_{i_A\in \N^A} \int_{\T^A}\d \xi_A \sum_{i_0\in \N}\int_{\T}\d \upsilon \left\|\varphi^{(A,B)}_{i_0,i_A}\p{\upsilon,\xi_A}\right\|_{\tilde{\mathcal{H}}_B}= \left\|\p{q^{(A\cup\{0\},B)}_{i}}_{i\in \N^{m+1}}\right\|_{E_{A\cup\{0\},B}(\mathcal{H})},\eeq
\beq \sum_{i_A\in \N^A} \int_{\T^A}\d \xi_A \sqrt{\sum_{i_0\in \N}\int_{\T}\d \upsilon \left\|\psi^{(A,B)}_{i_0,i_A}\p{\upsilon,\xi_A}\right\|_{\tilde{\mathcal{H}}_B}^2}= \left\|\p{q^{(A,B\cup\{0\})}_{i}}_{i\in \N^{m+1}}\right\|_{E_{A,B\cup\{0\}}(\mathcal{H})},\eeq
which proves that $q^{(A,B)}_i$ for $A\dot\cup B=[0,m]$ satisfy the desired inequality. 

\end{proof}

\section{Interpolation of $\p{U^1_m\cap\honeall\p{\T^\N}}\p{\ell^p}$}
It has been proved in \cite{xul1h1} that for any $X\subset L^1$, the property
\beq \label{eq:BPX}\p{X\p{\ell^1},X\p{\ell^2}}\text{ is }K\text{-closed in }\p{L^1\p{\ell^1},L^1\p{\ell^2}}\eeq
implies that $L^1/X$ has cotype 2 and every operator $L^1/X\to \ell^2$ is 1-summing. In \cite{diss}, we used the decomposition theorem to prove that $U^1_m$ satisfies \eqref{eq:BPX}. Here, we prove the same for $U^1_m\cap \honeall$, based on that reasoning. 
\begin{thm} The space $\p{U^1_m\cap \honeall}\p{\T^\N}\subset L^1\p{\T^\N}$ satisfies \eqref{eq:BPX}. \end{thm}
\begin{proof}[Proof (sketch).]We follow closely the proof of the analogous theorem in \cite{diss}. We will prove the $K$-closedness of $L^1(\R)$- and $L^2(\R)$-valued spaces instead. Let $f\in \p{U^1_m\cap \honeall}\p{\T^\N,L^1\p{\R}}+ \p{U^1_m\cap \honeall}\p{\T^\N,L^2(\R)}$. First, we check by rescaling in the second variable that it is enough to check the $K$-functional inequality for $t=1$. Then, by the fact that $f$ is $U^1_m$ in the variable in $\T^\N$, i.e. it has a representation 
\beq f(x)=\sum_{i_1<\ldots<i_m} f_{i_1,\ldots,i_m}\p{x_{i_1},\ldots,x_{i_m}}\eeq
where $f_{i}$ is $U^1_m\p{\T^m}$ and $f_{i}=\bigsqcup_k f_{i,k}$, we copy the old result
\beq \label{eq:kfunl1l1l1l2}K\p{f,1;L^1\p{L^1},L^1\p{L^2}}\simeq \int_{[0,1]^\N}\int_{\p{\T^\N}^m} \d s\d x \p{\sum_{i,k} \left|f_{i,k}\p{x^{(1)}_{i_1},\ldots,x^{(m)}_{i_m},s_k}\right|^2}^\frac12.\eeq 
By density we can assume that $f_{i,k}$ is a trigonometric polynomial in the last variable. Then, multiplying by a suitable character, using Theorem \ref{mainh12msum} and multiplying by the inverse of that character, we get a decomposition 
\beq f_{i,k}=\sum_{A\dot\cup B\subset [1,m+1]} a_{i,k}^{(A,B)}\eeq
into summands $a_{i,k}^{(A,B)}\in \honeall\p{\T^m,L^1[0,1]}$ satisfying the usual inequality 
\beq \sum_{A\dot\cup B=[1,m]}\left\|\bigsqcup_{i,k}a^{(A,B)}_{i,k}\right\|_{H^1_{\mathrm{all}}\p{\T^A,\ell^1\p{\N^A,H^2_{\mathrm{all}}\p{\T^B,\ell^2\p{\N^B}}}}}\lesssim \eqref{eq:kfunl1l1l1l2}\eeq
(the same reasoning can be use to recover the decomposition theorem in the Lebesgue case of \cite{diss} from the analytic one proved here). Then we proceed by taking \beq g_{i}=\bigsqcup_k \sum_{m+1\in A} a_{i,k}^{(A,B)},\quad h_{i}=\bigsqcup_k \sum_{m+1\in B} a_{i,k}^{(A,B)}.\eeq
This produces a decomposition of $f$ into summands belonging to $U^1_m\cap\honeall$ in the $\T^\N$ variable that satisfies
\beq \|g\|_{L^1\p{L^1}}+\|h\|_{L^1\p{L^2}}\lesssim \eqref{eq:kfunl1l1l1l2}\eeq
by the old calculation.
\end{proof}

\end{document}